\documentclass[11pt]{amsart}

\usepackage[margin=1in]{geometry}
\usepackage{amsmath,amsfonts,amssymb,amsthm,graphicx,cite,enumitem}
\usepackage[hidelinks]{hyperref}
\usepackage{xcolor}
\hypersetup{
    colorlinks,
    linkcolor={red!60!black},
    citecolor={blue!80!black},
    urlcolor={blue!80!black}
}
\graphicspath{{images/}}

\newcommand{\R}[1][]{{\mathbb{R}^{#1}}}

\newcommand{\haus}[1][\alpha]{\operatorname{V}_{#1}}

\numberwithin{equation}{section}

\theoremstyle{plain}

\newtheorem{thm}{Theorem}[section]
\newtheorem{lem}[thm]{Lemma}

\newtheorem{conj}[thm]{Conjecture}

\theoremstyle{definition}

\newtheorem{define}[thm]{Definition}
\newtheorem{ex}[thm]{Example}
\newtheorem{prob}[thm]{Problem}

\theoremstyle{remark}

\newtheorem{remark}[thm]{Remark}

\title{Non-symmetric Convex Polytopes and Gabor orthonormal bases}
\address{Department of Mathematics, San Francisco State University, San Francisco, California 94132}
\author{Randolf Chung}
\author{Chun-kit Lai}
\email{\href{mailto:university@jeongjh.com}{university@jeongjh.com}}
\email{\href{mailto:cklai@sfsu.edu}{cklai@sfsu.edu}}
\thanks{}
\subjclass[2010]{42C15,52B11}
\date{}
\keywords{Gabor orthonormal bases, Polytope, Spectral Sets, Tilings}

\begin{document}
\begin{abstract}
	In this paper, we show that  non-symmetric convex polytopes cannot serve as a window function to produce a Gabor orthonormal basis by any time-frequency sets.
\end{abstract}
\maketitle
\tableofcontents
\section{Introduction}

Let $\Omega$ be a subset of $\R[d]$ with $|\Omega|>0$ ($|\cdot|$ denotes the Lebesgue measure). If $\Gamma$ is a discrete subset of $\R[d]$, we write $E_\Gamma$ for the set of exponentials $\{e_\gamma(x):\gamma\in \Gamma\}$ where $e_\gamma(x):=e^{2\pi i \langle \gamma,x\rangle}$ for $x\in\R[d]$.
\begin{define}
	Let $g\neq 0$ be a function in $L^2(\R[d])$ and let $\Lambda=\{(t,\lambda): t,\lambda\in\R[d]\}$ be a discrete subset of $\R[2d]$. A {\it Gabor system} is a collection of translations and modulations of the function $g$ by $\Lambda$:
	\begin{equation}\label{Gaborsystem}
		\mathcal{G}(g,\Lambda):=\{e_\lambda(x)g(x-t): (t,\lambda)\in\Lambda\}.
	\end{equation}
    In particular, a measurable set $\Omega\subseteq \mathbb{R}^d$ is called a {\it Gabor orthonormal basis set (GONB set)} if ${\mathcal G}(|\Omega|^{-1/2}\chi_{\Omega},\Lambda)$ forms an orthonormal basis for $L^2(\R[2d])$.
\end{define}
We call $g$ and $\Lambda$ the {\it window function} and the {\it time-frequency set} respectively. $\Lambda$ is said to be {\it separable} if there exists sets ${\mathcal J}$ and $\Gamma$ on $\R[d]$ such that $\Lambda = {\mathcal J}\times \Gamma$.

\medskip

In recent years, determining a pair $(g,\Lambda)$ such that $\mathcal{G}(g,\Lambda)$ arises as a frame or orthonormal bases has received much attention and many important cases have been solved. Yet, there is still an abundance of mysteries and unexpected results within this classification (for example, see \cite{Grochenig2001,Grochenig2014}). Concerning the structure of GONB sets,  the following problem may give us some positive insight. It was recently proposed and studied by several authors \cite{Agora2017,Iosevich2017,Gabardo2015,Lai}.

\medskip
\begin{prob}[Fuglede-Gabor problem]
	Suppose $\Omega\subseteq \R^d$ is a GONB set. Then
	\begin{enumerate}[align=left]
		\item (Spectrality) there exists $\Gamma$ such that $E_\Gamma$ forms an orthonormal basis for $L^2(\Omega)$, and
		\item (Tiling) there exists a discrete set ${\mathcal J}$ such that $\R[d]$ is the almost disjoint union of $\Omega+t, t\in {\mathcal J}$. Equivalently,
		      $$\sum_{t\in{\mathcal J}} \chi_{\Omega}(x-t) = 1 \ \mbox{a.e.}$$
	\end{enumerate}
\end{prob}
In general, sets satisfying (1) and (2) are called {\it spectral sets} and {\it translational tiles} respectively.
Historically, the first related version of the Fuglede-Gabor problem was introduced in \cite{Liu2003}. They conjectured that if the window functions were compactly-supported and the time-frequency sets were separable, then the conclusion of the Fuglede-Gabor problem holds.  Due to the separability condition, the problem was settled by \cite{Dutkay2014} if the window was non-negative. Our interest is the non-separable case. In fact, considering standard objects such as the unit cube $[0,1]^d$, there exist uncountably many distinct (up to translation) non-separable time-frequency sets $\Lambda$ such that ${\mathcal G}(\chi_{[0,1]^d},\Lambda)$ forms an orthonormal basis if $d\ge 2$ (see \cite{Gabardo2015}).

\medskip

The Fuglede-Gabor problem is motivated by a related conjecture called the {\it spectral set conjecture}:

\begin{conj}[Spectral set conjecture]
	A set $\Omega$ is a spectral set if and only if it is a translational tile.
\end{conj}

This conjecture was introduced by Fuglede \cite{Fuglede1974} during his studies of extensions of commuting self-adjoint differential operators to dense subspaces of $L^2(\Omega)$. His conjecture was  disproven in one direction by Tao \cite{Tao2003} for $d\geq 5$ and then in both directions by Kolountzakis and Matolcsi \cite{Kolountzakis2006} for $d\geq 3$. Despite this however, the conjecture was verified in many significant cases including the following:

\begin{enumerate}
	\item $\Omega$ tiles by a lattice \cite{Fuglede1974},
	\item $\Omega$ is a union of two intervals on ${\mathbb R}^1$ \cite{Laba2000},
	\item $\Omega$ is a convex body with a point of positive Gaussian curvature \cite{Iosevich2001},
	\item $\Omega$ is a non-symmetric convex body \cite{Kolountzakis2000}.
\end{enumerate}

The first three cases have partially been resolved recently in the Fuglede-Gabor problem (see \cite{Lai} for  case (1), \cite{Agora2017} for case (2), and \cite{Iosevich2017} for case (3)). Each case used machinery similar to its Fuglede counterpart's, but due to the extra consideration of the set $\Omega\cap(\Omega+t)$, none of the cases were proven in their full generality.

\medskip

In this paper, we consider the fourth case with non-symmetric convex polytopes. Our main result is

\begin{thm}\label{main1}
	Let $\Omega$ be a non-symmetric convex polytope in ${\mathbb R}^{d}$. Then $\Omega$ is not a GONB set. In other words, there cannot exist a $\Lambda$ such that ${\mathcal G}(|\Omega|^{-1/2}\chi_{\Omega},\Lambda)$ forms an orthonormal basis.
\end{thm}

We are unable to generalize the proof in \cite{Kolountzakis2000} to obtain a more general result for convex bodies (see Remark \ref{remark_end}). Instead, we will follow a similar approach by Greenfeld and Lev \cite[Theorem 3.1]{Greenfeld2017} (originally from \cite{Kolountzakis2002}). To fully utilize the same line of thought, we will first consider the intersection of the polytopes $\Omega$ and its translate $\Omega+t$. We must assure that for a sufficiently small vector $t$, $\Omega\cap (\Omega+t)$ will remain non-symmetric with the $(d-1)$-volume of their facets staying continuous (Theorem \ref{epsi}). After that, we apply an analogous argument from Greenfeld-Lev twice on the frequency and time axes to obtain a similar contradiction.

\section{Lemmas on polytopes}

In this section, we study the structure of convex polytopes. Main references will be taken from \cite{Gruber2007,Schneider2013}. Let us recall some terminology.

\medskip

Let $\haus$ be the $\alpha$-dimensional volume function on $\R[d]$. A (closed) half-space $H$ is defined by $\{x\in \R[d]: \langle a,x\rangle \le b\}$ where $a$ is the normal vector to $H$.  A \textit{convex polyhedron} is a finite intersection of closed half-spaces; thus, a convex polyhedron $\Omega$ is a closed set admitting a {\it half-space representation}
\begin{equation}\label{polyhedra}
	\Omega = \{x\in{\mathbb R}^d: \langle a_i,x\rangle\le b_i, \ \forall i=1,...,n\}= \bigcap_{i=1}^n H_i,
\end{equation}
where $H_i = \{x\in{\mathbb R}^d: \langle a_i,x\rangle\le b_i\}.$ A {\it facet} $F_i$ of $\Omega$ is the intersection of $\Omega$ with the boundary of a half-space in its half-space representation; namely, $F_i = (\partial H_i)\cap \Omega$ such that $\haus[d-1](F_i)>0$.

\medskip

A {\it convex polytope} is the convex hull of finitely many points. It is well-known that a convex polytope is equivalent to a bounded polyhedron. A convex polytope is {\it (centrally) symmetric} if there exists a point $x\in\R[d]$ such that
$$
x-\Omega = \Omega-x.
$$
If $F=(\partial H)\cap \Omega$ is a facet of $\Omega$, then $F'=(\partial H')\cap \Omega$ is the {\it parallel} of $F$ if $(\partial H)\cap(\partial H')=\emptyset$ (i.e. $H$ and $H'$ share unit normal vectors in opposing directions). By convention, we take $\emptyset$ to be the parallel facet of $F$ if a parallel facet does not exist. The following theorem fully characterizes symmetric convex polytopes in terms of parallel facets and volume (see \cite[Corollary 18.1]{Gruber2007}):
\begin{thm}[Minkowski's Theorem]
	A convex polytope is symmetric if and only if for every facet $F\subset \Omega$, there exists a parallel facet $F'$ such that $\haus[d-1](F')=\haus[d-1](F)$.
\end{thm}
\medskip

Let ${\mathcal C}:={\mathcal C}[\R[d]]$ be the set of compact convex sets on $\R[d]$, and let $B_\delta(x)$ be the open ball of radius $\delta$ centered at $x$. We will denote by ${\mathcal P}:={\mathcal P}[\R[d]]$ the set of all polytopes in ${\mathcal C}$.

\medskip

For any $E,F\in{\mathcal C}$, the {\it Hausdorff metric} of $E$ and $F$ is defined as
$$
d_{H}(E,F) = \inf\{\delta: E\subset F^{\delta} \ \mbox{and} \ F\subset E^{\delta}\},
$$
where $E^{\delta} := \bigcup_{x\in E} B_\delta(x)$ and similarly for $F^{\delta}$. The metric space $(\mathcal C,d_{H})$ is complete.

\medskip

Now we remark that in general the volume function is not continuous for general compact sets.

\begin{ex}
	Let $T_0:=[v_1;v_2;v_3]$ denote a 2-simplex in $\R[2]$ with unit side lengths, and let $T_n:=[v_1;v_2;(1/n)v_3+(1-1/n)v_1]$ for $n>0$. $T_n$ converges to the line segment $L$ joining $v_1,v_2$, but of particular interest, we see the \textit{non-convex} sequence $\partial T_n$ converges to the line segment $L$ joining $v_1,v_2$. By triangle inequality, $\haus[1](\partial T_n)\geq 2$ while $\haus[1](L)=1$, so $\haus[1](\partial T_n)$ cannot converge to $\haus[1](L)$.
\end{ex}

\medskip

Nonetheless, $\haus[d-1]$ is continuous on ${\mathcal C}$. A quick way to see this can be found in \cite[p.104-105]{Gruber2007}. In summary, up to a constant, $\haus[d-1]$ computes the surface area of a facet and according to \cite[p.104-105]{Gruber2007},
$$
\haus[d-1](C) = k_dW_1(C),
$$
where $W_1$ is the quermassintegral of $C$ and $k_d>0$ is some constant dependent on the dimension. It is thus a continuous function on $({\mathcal C}, d_{H})$ (by \cite[Theorem 6.13]{Gruber2007}); hence,
\begin{equation}\label{continuity}
	\lim_{n\rightarrow\infty}{d_H}(E_n,F)  \ \Rightarrow \ \lim_{n\rightarrow\infty}\mbox{V}_{d-1}(E_n)  =\mbox{V}_{d-1}(F).
\end{equation}

\medskip

Our goal now is to show that $\Omega_t:=\Omega\cap (\Omega+t)$ is non-symmetric for $t$ small if $\Omega$ is non-symmetric.

\medskip

\begin{lem}\label{min}
	Let $t\in\R[d]$, and let $\Omega$ be given by \eqref{polyhedra}. Then $\Omega_t$ admits a representation
	\begin{equation}\label{eqnmin}
		\Omega_t=\bigcap_{i=1}^{n}M_i
	\end{equation}
	where
	$$
	M_i:=M_i(t)=\{x\in{\mathbb R}^d: \langle a_i,x\rangle\le m_i\} \quad \mbox{and}  \quad m_i:=\min\{b_i, b_i+\langle a_i,t\rangle\}.
	$$
\end{lem}
\begin{proof}
	Let $\Omega=\bigcap_{i=1}^n H_i$ where $H_i=\{x:\langle a_i,x\rangle\leq b_i\}$. We have
	$$
	H_i+t=\{x+t:\langle a_i,x\rangle\leq b_i\}=\{x:\langle a_i,x-t\rangle\leq b_i\}=\{x:\langle a_i,x\rangle\leq b_i+\langle a_i,t\rangle\}
	$$
	Let $m_i$ be defined as above.  Then it follows immediately that
	$$
	H_i\cap (H_i+t) =\{x:\langle a_i,x\rangle\leq m_i\} = M_i
	$$
	Since
	$$
	\Omega_t=\Omega\cap(\Omega+t) = \left(\bigcap_{i=1}^n H_i\right)\cap \left(\bigcap_{i=1}^n H_i+t\right)  = \bigcap_{i=1}^n (H_i \cap (H_i+t))=\bigcap_{i=1}^n M_i,
	$$
	this implies \eqref{eqnmin}.
\end{proof}

\medskip

The following lemma shows that the facet in $\Omega_t$ converges to the original facet $\Omega$ in Hausdorff metric.

\begin{lem}\label{continter}
	Let $\Omega\in\mathcal{P}$, and let $F=(\partial H)\cap \Omega$ be a facet of $\Omega$. Write $H = \{x\in{\mathbb R}^d: \langle a,x\rangle\le b\}$ and let $M(t)$ be as defined in Lemma \ref{min} for $H$. Then the facets $F(t)=(\partial M(t))\cap \Omega_t$ converges to $F$ as $t\to0$.
\end{lem}
\begin{proof}
	By \cite[Theorem 1.8.8]{Schneider2013}, a sequence of compact convex sets $K_i$ converges to $K$ if and only if
\begin{enumerate}
  \item every point $x\in K$ is the limit of some sequence of points $\{x_i\}, x_i\in K_i$.
  \item for any convergent sequences $(x_{i_j})$ with $x_{i_j}\in K_{i_j}$, the limit of $x_{i_j}$ belongs to $K$
\end{enumerate}
$(1)$ is clear since $x+t\to x$ as $|t|\to 0$ and $x+t\in F(t)$. For $(2)$, choose any convergent sequence $(x_{t_i})$ with $x_{t_i}\in F(t_i)$ and denote its limit by $x$. Then Lemma \ref{min} implies that
	$$
	\partial M(t) = \{x:  \langle a,x\rangle =\min\{b,b+\langle a,t\rangle\}\}.
	$$
	Now, $x_{t_i}\in \partial M(t_i)$, so $\langle a,x_{t_i}\rangle =  \min\{b,b+\langle a,t_i\rangle\}$. But $t_i$ converges to 0 by the continuity of $\langle \cdot,\cdot\rangle$ and $\min$, so $\langle a,x\rangle =  b$. In other words, $x\in \partial H$. On the other hand, $x\in \Omega_t = \Omega\cap (\Omega+t)\subset \Omega$, so $x\in F$. This completes the proof.
\end{proof}

\begin{thm}\label{epsi}
	Suppose $\Omega$ is a non-symmetric polytope. Then there exists $\epsilon>0$ such that for all $|t|\leq \epsilon$, $\Omega_t$ is non-symmetric. More specifically, given a non-symmetric facet $F$ in $\Omega$, $F(t)$ is a non-symmetric facet for $\Omega_t$ for $|t|\leq \epsilon$.
\end{thm}
\begin{proof}
	It suffices to show the second statement since then the first statement will follow from Minkowski's Theorem.
	
	\medskip
	
	Let $F$ be a non-symmetric facet of $\Omega$, and choose a facet $F'$ parallel to $F$ with $\haus[d-1](F)\neq \haus[d-1](F')$. By Minkowski's Theorem, such a facet is guaranteed to exist. Define $$
  V(t):=|\haus[d-1](F(t))-\haus[d-1](F'(t))|.$$ By Lemma \ref{continter} and \eqref{continuity}, $\haus[d-1](F(t))$ and $\haus[d-1](F'(t))$ are continuous at $0$, hence $V(t)$ is continuous at $0$. So
$$
V(0)=|\haus[d-1](F(0))-\haus[d-1](F'(0))|=|\haus[d-1](F)-\haus[d-1](F')|>0,
$$
thus we can choose some $\epsilon>0$ such that $V(t)>0$ for $|t|<\epsilon$. Choosing $\epsilon$ smaller, this holds true for the compact ball $|t|\leq \epsilon$. Thus
$$
|\haus[d-1](F(t))-\haus[d-1](F'(t))|=V(t)>0.
$$
This complete the proof.
\end{proof}
We remark that the condition on $t$ cannot be removed. The following example shows that intersection of a non-symmetric polytope and its translate may become symmetric for sufficiently large $t$.

\begin{ex}\label{ex2.6}
	Let $\Omega$ be the polytope with five edges and vertices given by $(0,0), (2,0), (2,2), (1,2),$ and $ (0,1)$. It is a square with a top left-hand corner removed and it is clearly non-symmetric.  Consider $t = (-1,-1)$. Then $\Omega\cap (\Omega+t)$ becomes a square with vertices $(0,0),(1,0),(0,1),(1,1)$, so it is symmetric.
\end{ex}

\section{Proof of the main theorem}
Let $\Omega$ be the convex polytope on $\R[d]$. We denote by $\sigma_F(x)$ the surface measure on the facet $F$ of $\Omega$.  Let $n_{F}$ denote the outward unit normal to the facet $F$ on $\Omega$. From Lemma \ref{min}, the corresponding  facet $F(t)$ of $\Omega_t = \Omega \cap (\Omega+t)$ shares the same normal vector. The following lemma is a variant of Greenfeld-Lev \cite[Lemma 2.7]{Greenfeld2017} (the case $t=0$). We show that the lower order term can be bounded, independent of $t$.

\begin{lem}\label{bound}
	Let $A(t)$ be a facet of $\Omega_t$, and let $B(t)$ be the parallel facet to $A(t)$ of $\Omega_t$ with outward unit normals $e_1$ and $-e_1$. Then there exists $\omega:=\omega_\Omega>0$, independent of $t$, such that in the cone
	$$
	C(\omega):=\{\lambda\in \R[d]:|\lambda_j|\leq \omega|\lambda_1|\mbox{ for all }2\leq j\leq d\},
	$$
	we have
	\begin{equation}\label{eqn3}
		-2\pi i\lambda_1\widehat{\chi}_{\Omega_t}(\lambda)=\widehat{\sigma}_{A(t)}(\lambda)-\widehat{\sigma}_{B(t)}(\lambda)+G_t(\lambda)
	\end{equation}
	with
	$$
	|G_t(\lambda)| \le \frac{C}{|\lambda_1|}
	$$
	for some constant $C>0$, independent of $t$.
\end{lem}
\begin{proof}
	By divergence theorem (see \cite[Lemma 2.4]{Greenfeld2017}),
	$$
	-2\pi i\lambda_1\widehat{\chi}_{\Omega_t}(\lambda) = \widehat{\sigma}_{A(t)}(\lambda)-\widehat{\sigma}_{B(t)}(\lambda) +\sum  \langle e_1, n_F\rangle\widehat{\sigma}_{F(t)}(\lambda)
	$$
	where the sum is over all facets $F(t)$ of  $\Omega_t$ except $A(t)$ and $B(t)$. Define $G_t(\lambda)$ to be the sum. By \cite[Lemma 2.6]{Greenfeld2017},
	\begin{equation}\label{eqsigma}
		|\widehat{\sigma}_{F(t)} (\lambda)|\leq \frac{\haus[d-2](\partial F(t))}{2\pi}\cdot\frac{|\lambda|^{-1}}{|\sin \theta_{\lambda, n_F}|} \le \frac{\haus[d-2](\partial F)}{2\pi}\cdot\frac{|\lambda|^{-1}}{|\sin \theta_{\lambda, n_F}|}
	\end{equation}
	where $\theta_{\lambda,n_F}$ is the angle between $\lambda\in\R[d]\backslash \{0\}$.  The second inequality follows from the fact that the facet $F(t)$ is either empty, a subset of the facet $F$, or a subset of the facet $F+t$, so $\haus[d-2](\partial F(t)) \le \haus[d-2](\partial F)$. If $\omega$ is sufficiently small, then for $\lambda\in C(\omega)$, $\theta_{\lambda,n_F}$ is bounded away from 0 and $\pi$ for all $n_F$, so inside the cone $C(\omega)$, summing up all $F$ in (\ref{eqsigma}) shows $G_t(\lambda)$ is bounded by $C|\lambda|^{-1}$ as $|\lambda_1|\to\infty$. As $n_F$ does not depend on $t$, $C$ does not depend on $t$.
\end{proof}

\medskip

We now return to the main problem. Let $g \in L^2({\mathbb R}^d)$. The {\it short time Fourier transform} (STFT) is defined by
$$
V_gg (t,\lambda) := \int g(x)\overline{g(x-t)}e^{-2\pi i \langle\lambda, x\rangle} dx.
$$
If $g = |\Omega|^{-1/2}\chi_{\Omega}$, we have
\begin{equation}\label{V_gg}
	V_gg (t,\lambda) = |\Omega|^{-1}\widehat{\chi}_{\Omega\cap (\Omega+t)} (\lambda) = |\Omega|^{-1}\widehat{\chi}_{\Omega_t} (\lambda).
\end{equation}
We observe a Gabor system ${\mathcal G}(g,\Lambda)$ forms an orthonormal basis if and only if the following holds
\begin{enumerate}
	\item (Mutual Orthogonality) $\Lambda - \Lambda \subset \{(t,\lambda): V_gg(t,\lambda) =0\}$, and
	\item (Completeness) ${\mathcal G}(g,\Lambda)$ is complete in $L^2({\mathbb R}^d)$.
\end{enumerate}
(see \cite{Gabardo2015,Agora2017} for a complete derivation). Furthermore, if ${\mathcal G}(g,\Lambda)$ forms an orthonormal basis, due to the continuity of $V_gg$ at the origin, $\Lambda$ must be {\it uniformly discrete}, i.e. there exists $\delta>0$ such that every ball of radius $\delta$ intersects $\Lambda$ at at most one point. On the other hand, $\Lambda$ is {\it relatively dense} in ${\mathbb R}^{2d}$ in the sense that there exists $R>0$ such that any balls of radius $R$ must intersect $\Lambda$ since the density of $\Lambda$ on ${\mathbb R}^{2d}$ must equal one (see \cite{Ramanathan1995}).

\medskip

Let $S(r)=\{te_1+w:t\in\mathbb{R}, w\in\mathbb{R}^{d}, |w|< r\}$ be the cylinder along the $x_1$-axis.

\begin{lem}\label{nonzero}
	Suppose $\Omega$ is a non-symmetric convex polytope on $\R[d]$ and $g = |\Omega|^{-1/2}\chi_{\Omega}$. There exist  $\epsilon>0$, $R>0$ and $\delta>0$, all independent of $t$, such that
	$$
	V_gg (t,\lambda)\neq 0,\quad \forall\lambda\in S(2\delta)\setminus B_{R}(0)
	$$
	for all $|t|<\epsilon$,
\end{lem}
\begin{proof}
	Take $\epsilon>0$ from Theorem $\ref{epsi}$, and consider $|t|\leq \epsilon$. Let $A(t)$ be the non-symmetric facet of $\Omega_t$ and let $B(t)$ be its parallel facet. Using an affine transformation, assume $A(t)$ and $B(t)$ lie on the hyperplanes $\{x_1=0\}$ and $\{x_1=1\}$ respectively, and let $\eta:=\min_{|t|\leq\epsilon} |\haus[d-1](A(t))-\haus[d-1](B(t))|$. By Theorem \ref{epsi}, $\eta>0$.

    We have
	$$
	\widehat{\sigma}_{A(t)}(\lambda) = \widehat{\chi}_{A(t)}(\lambda_2,...,\lambda_d)                    \quad\mbox{and}\quad\widehat{\sigma}_{B(t)}(\lambda) = e^{2\pi i\lambda_1}\widehat{\chi}_{B(t)}(\lambda_2,...,\lambda_d)
	$$
	where $\chi_{B(t)}$ and $\chi_{A(t)}$ are the characteristic functions of the orthogonal projections of $B(t)$ and $A(t)$ onto $(x_2,...,x_d)$ respectively. Moreover, we can deduce
	\begin{align*}
    \widehat{\chi}_{A(t)}(0)&=\haus[d-1](A(t))\\
    \widehat{\chi}_{A(t)}-\widehat{\chi}_{A(0)}&=\widehat{\chi}_{A(t)\Delta A(0)}
	\end{align*}
    where $\Delta$ is the symmetric difference. This implies that
    $$
    |\widehat{\chi}_{A(t)}(\lambda')-\widehat{\chi}_{A(0)}(\lambda')|\leq \haus[d-1](A(t)\Delta A(0)) \to 0\mbox{ as }t\to 0,  \ \forall \lambda'\in{\mathbb R}^{d-1}
    $$
    so $\widehat{\sigma}_{A(t)}$ converges uniformly to $\widehat{\sigma}_{A(0)}$ on ${\mathbb R}^{d-1}$. Similarly, $\widehat{\sigma}_{B(t)}$ converges uniformly to $\widehat{\sigma}_{B(0)}$.

	Thus, by uniformity, we can choose $\delta>0$, independent of $t$, such that $$|\widehat{\sigma}_{A(t)}(\lambda)-\widehat{\sigma}_{B(t)}(\lambda)|\geq \eta$$ in the cylinder $S(2\delta)$. Using \eqref{V_gg} and Lemma \ref{bound}, we can choose $\omega>0$ and $C>0$, independent of $t$, such that
	$$
	2\pi |\Omega||\lambda_1||V_gg(t,\lambda)| \ge \eta - |G_t(\lambda)| \ge \eta - \frac{C}{|\lambda_1|}
	$$
	in the cone intersection $C(\omega)\cap S(2\delta)$. Taking $R$ large so that $S(2\delta)\setminus B_{R}(0)\subseteq C(\omega)\setminus B_{R}(0)$ and
	$$
    \eta - \frac{C}{|\lambda_1|}>0  \ \mbox{on} \ S(2\delta)\setminus B_{R}(0),
    $$
	we see that
	$$
	V_gg (t,\lambda)\neq 0,\quad \lambda\in S(2\delta)\setminus B_{R}(0)
	$$
	for any $|t|<\epsilon$. Since the constant $C$ and $\omega$ are taken independently of $t$, $R$ is independent of $t$, so we are done.
\end{proof}
We now give the proof for Theorem $\ref{main1}$.
\begin{proof}[Proof of Theorem \ref{main1}]
	We argue by contradiction. Suppose ${\mathcal G}(g,\Lambda)$ forms a Gabor orthonormal basis, and let $\epsilon, \delta, R$ be as defined in the previous lemma.
	
	\medskip
	
	{\it Claim:} For any $\tau,x \in\R[d]$, $\operatorname{card}(\Lambda \cap [(B_{\epsilon/2}(x)\times (S(\delta)+\tau)])<\infty$ where $\operatorname{card}(\cdot)$ denotes cardinality.
	
	\medskip
	
	Suppose not. As $\Lambda$ is uniformly discrete,  one can find  $v=(t,\lambda)$ and $v'=(t',\lambda')\in\Lambda \cap [(B_{\epsilon/2}(\nu)\times (S(\delta)+\tau)]$ with $|\lambda-\lambda'|>R$. But $|t-t'|<\epsilon$ and $\lambda-\lambda'\in S(2\delta)$, Lemma \ref{nonzero} tells us that we must have
	$$
	V_gg (t-t',\lambda-\lambda')\neq 0.
	$$
	This contradicts the mutual orthogonality of $\Lambda$. Thus, $|\lambda'-\lambda|\le R$, otherwise, $\lambda'-\lambda\in S(2\delta)\setminus B_{R}$ which implies $V_gg (t,\lambda)\neq 0$, a contradiction to the mutual orthogonality. This establishes the claim.
	
	\medskip
	
	Now since $\Lambda$ is a relatively dense set, there is a radius $\delta^*>0$ such that every $2d$-ball of radius $\delta^*$ non-trivially intersects $\Lambda$. Consider the set $B_{\delta^*}^d(0)\times S(\delta^*)$ ($d$ denotes the $d$-dimensional ball) covered by finitely many cylinders $B_{\epsilon/2}^d(\nu_i)\times (S(\delta)+\tau_j), 1\leq i,j\leq N$. Then $\operatorname{card}(\Lambda\cap [B_{\delta^*}^d(0)\times S(\delta^*)])<\infty$. However, this implies that  $B_{\delta^{\ast}}^d(0)\times S(\delta^*)$ contains a $2d$-ball of radius $\delta^*$ that does not intersect $\Lambda$, a contradiction to the relative density. It follows that such $\Lambda$ does not exist and our proof is complete.
\end{proof}

\begin{remark}\label{remark_end}
	There is an approach to the Fuglede conjecture used in \cite{Kolountzakis2000} which considers the Fourier transform of the function $f = |\widehat{\chi_{\Omega}}|^2$. This transform is equal to $\chi_{\Omega}\ast\chi_{-\Omega}$, so $\widehat{f}$ has compact support, allowing the use of \cite[Theorem 2]{Kolountzakis2000} to obtain a conclusion about the support of the Fourier transform of $\delta_{\Gamma}$ (as a tempered distribution). Considering $f = |V_gg|^2$ on ${\mathbb R}^{2d}$ with $g = \chi_{\Omega}$,
	$$
	(|V_gg|)^{\widehat{}} (t,\xi) = V_gg (\xi,-t)
	$$
	(see \cite[Equation (11) in page 873]{Grochenig2014}), there is no compactly supported Fourier transform (since the time side is unbounded), so the method in \cite{Kolountzakis2000} cannot be realized without some non-trivial adjustment.
\end{remark}
\section{Acknowledgements}
This work was an undergraduate research project in 2015-16 supported by the Office of Research and Sponsorship Programs (ORSP) in San Francisco State University (Grant No. ST 659). Both authors would like to thank the support from ORSP making this research possible. The authors would also like to thank professor Joseph Gubeladze for providing Example 2.6.

\bibliographystyle{customalpha}
\bibliography{References}
\end{document}